\documentclass[final,a4paper]{siamltex}

\usepackage{amsmath,amsfonts,amssymb,mathtools}
\usepackage{graphicx,subfigure}
\usepackage{braket,dsfont,url,enumerate}

\usepackage{pdfpages} 
\usepackage{cmap}
\usepackage{hyperref}
\usepackage{color,dsfont,braket}

\newtheorem{assumption}{Assumption}

\newtheorem{remark}[theorem]{Remark}

\newcommand{\V}{\mathbb{V}}

\newcommand{\C}{\mathbb{C}}

\newcommand{\pa}{\partial}
\newcommand{\eps}{\varepsilon}

\newcommand{\Om}{\Omega}

\newcommand{\IOm}{I\times \Om}

\newcommand{\norm}[1]{\lVert#1\rVert}
\newcommand{\ltwonorm}[1]{\lVert#1\rVert_{L^2(\Omega)}}

\newcommand{\abs}[1]{\lvert#1\rvert}

\newcommand{\Ppol}[1]{\mathcal{P}_{#1}}

\newcommand{\R}{\mathbb{R}}
\newcommand{\lk}{\ln{\frac{T}{k}}}

\newcommand{\Xk}{X^0_k}
\newcommand{\Xkh}{X^{0,r}_{k,h}}
\newcommand{\Cell}{\alpha}

\definecolor{darkred}{rgb}{.7,0,0}

\definecolor{green}{rgb}{0,0.7,0}

\begin{document}

\title{Discrete maximal parabolic regularity  for Galerkin finite element methods for non-autonomous parabolic problems}

\author{
Dmitriy Leykekhman\footnotemark[2]
\and
Boris Vexler\footnotemark[3]
}

\pagestyle{myheadings}
\markboth{DMITRIY LEYKEKHMAN AND BORIS VEXLER}{Discrete maximal parabolic regularity}

\maketitle

\renewcommand{\thefootnote}{\fnsymbol{footnote}}
\footnotetext[2]{Department of Mathematics,
               University of Connecticut,
              Storrs,
              CT~06269, USA (dmitriy.leykekhman@uconn.edu). The author was partially supported by NSF grant DMS-1115288. }

\footnotetext[3]{Technical University of Munich,
Chair of Optimal Control,
Center for Mathematical Sciences,
Boltzmannstra{\ss}e 3,
85748 Garching by Munich, Germany, (vexler@ma.tum.de). }

\renewcommand{\thefootnote}{\arabic{footnote}}


\begin{abstract}
The main goal of the paper is to establish time semidiscrete  and 
space-time fully discrete maximal parabolic regularity for the lowest 
order time discontinuous Galerkin solution  of linear parabolic 
equations with time-dependent coefficients. Such estimates have many 
applications. As one of the applications  we establish  best 
approximations type results with respect to the $L^p(0,T;L^2(\Omega))$ 
norm for $1\le p\le \infty$.
\end{abstract}



\section{Introduction}
Let $\Om$ be a Lipschitz domain in $\mathbb{R}^d$, $d\geq 1$ and $I=(0,T)$.
We consider the following second order parabolic partial differential equation with time-dependent coefficients,
\begin{equation}\label{eq: heat equation}
\begin{aligned}
\pa_tu(t,x)+A(t,x) u(t,x) &= f(t,x), & (t,x) &\in \IOm,\;  \\
    u(t,x) &= 0,    & (t,x) &\in I\times\pa\Omega, \\
   u(0,x) &= u_0(x),    & x &\in \Omega,
\end{aligned}
\end{equation}
with the right-hand side $f \in L^p(I;L^2(\Omega))$ for some $1\le p\le \infty$ and $u_0\in L^2(\Om)$,
where the time-dependent elliptic operator is given by the formal expression
\begin{equation}\label{eq: elliptic}
A(t,x)u(t,x)=-\sum_{i,j=1}^d\partial_j(a_{ij}(t,x)\partial_i u(t,x))
\end{equation}
with $a_{ij}(t,x)\in L^\infty(I\times \Om)$ for $i,j=1,\dots,d$ satisfying  $a_{ij}=a_{ji}$ and the uniform ellipticity property
\begin{equation}\label{eq: uniform ellipticity}
\sum_{i,j=1}^da_{ij}(t,x)\xi_i\xi_j\geq \alpha|\xi|^2, \quad \forall \xi\in \mathbb{R}^d \quad \text{and a.e.}\ (t,x)\in I\times \Om,
\end{equation}
for some constant $\alpha>0$. We also assume that the coefficients $a_{ij}(t,x)$ are continuous in $t$ for almost all $x \in \Omega$ and that the following condition holds:
\begin{equation}\label{eq: condition on time dependent coefficients}
    |a_{ij}(t_1,x)-a_{ij}(t_2,x)|\le \omega(|t_1-t_2|),\quad 1\le i,j\le d
\end{equation}
for all $t_1, t_2 \in \bar I$  and almost all $x\in \Om$, where $\omega \colon [0,T]\to [0,\infty)$ is a nondecreasing function such that
\begin{equation}\label{eq: condition on omega}
\frac{w(t)}{t^{\frac{3}{2}}} \text{ is non-increasing on } (0,T] \text{ and}\quad  \int_0^T\frac{\omega(t)}{t^{\frac{3}{2}}}dt<\infty,
\end{equation}
see \cite{HaakBH_OuhabazEM_2015} for a similar assumptions. This assumption is fulfilled, for example, if $a_{ij}$ is H{\"o}lder continuous with exponent $\frac{1}{2}+\eps$ in time and uniformly continuous in space. 

The maximal parabolic regularity for $u_0\equiv 0$ says that there exists a constant $C$ such that for $f\in L^p(I;L^s(\Omega))$,
\begin{equation}\label{eq: continuous maximal parabolic}
\norm{\pa_t u}_{L^p(I;L^s(\Omega))} + \norm{A(\cdot) u }_{L^p(I;L^s(\Omega))} \le C \, \norm{f}_{L^p(I;L^s(\Omega))}, \quad 1<s,p<\infty.
\end{equation}
For time-independent coefficients the above result is well understood  
\cite{CannarsaP_VespriV_1986, CoulhonT_DuongXT_2000, DenkR_HieberM_PrussJ_2003, Haller-DintelmannR_RehbergJ_2009, HieberM_PrussJ_1997}, however for time-dependent coefficients it is still an active area of research \cite{AuscherP_EgertM_2016, DierD_2015, DisserK_ElstAFM_RehbergJ_2017, HaakBH_OuhabazEM_2015}.
The maximal parabolic regularity is an important analytical tool and has a number of applications, especially to nonlinear problems and optimal control problems when sharp regularity results are required (cf., e.g.,~\cite{HombergD_MeyerC_RehbergJ_RingW_2009, KrumbiegelK_RehbergJ_2013a,KunischK_PieperK_VexlerB_2014, LeykekhmanD_VexlerB_2013, LeykekhmanD_VexlerB_2016c}).

The main goal of this paper is to establish similar maximal parabolic regularity results for time semidiscrete discontinuous Galerkin solutions as well as for fully discrete Galerkin approximations. Such results are very useful, for example, in a  priori error estimates  and essential in obtaining error estimates where the spatial mesh size $h$ and the time steps $k$ are independent of each other (cf.~\cite{LeykekhmanD_VexlerB_2017a, LeykekhmanD_VexlerB_2016a}).

Previously in \cite{LeykekhmanD_VexlerB_2016b} we established the corresponding discrete maximal parabolic regularity for discontinuous Galerkin time schemes of arbitrary order for autonomous problems. The extension to non-autonomous problems is not straightforward, especially for the critical values of $s,p=1$ or $s,p=\infty$. 
In this paper, we investigate the maximal parabolic regularity for $s=2$ and arbitrary $1\le p\le \infty$ for the lowest order time discontinuous Galerkin (dG$(0)$) methods, which can be considered as modified backward Euler (BE) method. 
The main difference between dG($0$) and BE methods  lies in the way  the time-dependent coefficients and the right-hand side are approximated. In dG($0$) formulation they are approximated by averages over each subinterval $I_m$ (see the details below) while in BE methods they are evaluated at time nodes.
 As a result, dG($0$)  approximations  are weakly consistent, i.e. satisfy the Galerkin orthogonality relation, see Section \ref{sec: preliminaries} for details.  

Parabolic problems with time-dependent coefficients are important, for example for analyzing quasi-linear problems. Over the years, there have been a considerable number of publications devoted to various numerical methods for   problems with time-dependent coefficients \cite{ChrysafinosK_WalkingtonNJ_2006a, DobrowolskiM_1978, GonzalezC_PalenciaC_1998, HippD_HochbruckM_OstermannA_2014, HuangMY_ThomeeV_1981, KarakashianO_1986, KeelingS_1989, LuskinM_RannacherR_1982, LuskinM_RannacherR_1982b, SammonP_1983, SammonP_1978}. The publication~\cite{ChrysafinosK_WalkingtonNJ_2006a} is the most relevant to our presentation since it treats discontinuous Galerkin methods. However, none of the above publications addresses the question of the discrete maximal parabolic regularity and the techniques used in those papers are not immediately applicable for establishing such results even for $p=2$.

Time discrete maximal parabolic regularity (sometimes called well-posedness or coercivity property in the literature) have been investigated in a number of publications for various time schemes, \cite{AshyralyevA_PiskarevS_WeisL_2002, AshyralyevA_SobolevskiiPE_1994, GuidettiD_2007, GuidettiD_PiskarevS_1998, GuidettiD_PiskarevS_1999, KovacsB_LiBLubichC_2016, LeykekhmanD_VexlerB_2016b}. However, all the above mentioned publications are dealing with the autonomous case. The extension to non-autonomous situation is not easy. We are only aware of the publication \cite{LiB_SunW_2017a}, where the discrete maximal parabolic regularity is established for  problems with time-dependent coefficients for the backward Euler  method. Although the results in \cite{LiB_SunW_2017a} are similar in nature, there are some significant differences. The results in \cite{LiB_SunW_2017a} require $a_{ij}(t,x)\in W^1_\infty(I\times \Om)$, smoothness of $\Om$ and treat only uniform time steps, but they are valid in $L^s(\Om)$ norms for $1<s<\infty$. Our results, on the other hand,  require only a H\"{o}lder continuity of $a_{ij}(t,x)$ in $t$ and $L^\infty$ in space, allow $\Om$ to be merely Lipschitz and treat variable time steps, but are valid only in $L^2(\Om)$ norm in space. Moreover,  the discrete maximal parabolic regularity in  \cite{LiB_SunW_2017a} is shown in $l^p(I;L^s(\Om))$  norm for $1<p,s<\infty$ and since their proof requires Gr\"{o}nwall's inequality, the argument can not be naturally extended to the critical values of $p=1$ and $p=\infty$ even with the expense of the logarithmic term.  We  establish our result by a completely different argument, including fully discrete Galerkin approximations, in $L^p(I;L^2(\Om))$ norm for any $1\le p\le \infty$. For our future applications the inclusion of the critical values of $p=1$ and $p=\infty$ is essential for error estimates in $L^\infty(I;L^2(\Om))$  norm. We also want to mention that we went through some technical obstacles in order to incorporate variable times steps. In the case of uniform time steps many arguments can be significantly simplified.

Our presentation is inspired by \cite{HaakBH_OuhabazEM_2015}, where the maximal parabolic regularity was established for continuous problems for $s=2$ and $1<p<\infty$ with rather weak assumptions on $A$.
 Thus, in particular, we show for dG($0$) method the semidiscrete solution $u_k$ on any time level $m$ for $u_0=0$  and $f\in L^\infty(I; L^2(\Om))$ satisfies
\begin{equation}\label{eq: maximal parabolic dgr intro combined}
\begin{aligned}
\|A_{k,m} u_{k,m}\|_{L^\infty(I_m;L^2(\Om))}+\left\|\frac{[u_{k}]_{m-1}}{k_m} \right\|_{L^2(\Om)}\le C\lk\|f\|_{L^\infty(I;L^2(\Om))}. 
\end{aligned}
\end{equation}
For $p=1$ with $u_0\in L^2(\Om)$  and $f\in L^1(I; L^2(\Om))$  we also obtain
\begin{equation}\label{eq: maximal parabolic dgr L1}
\sum_{m}\left(\|A_{k,m} u_{k,m}\|_{L^1(I_m;L^2(\Om))}+\|[u_{k}]_{m-1}\|_{L^2(\Om)}\right)\le C\lk\left(\|f\|_{L^1(I;L^2(\Om))}+\|u_0\|_{L^2(\Om)}\right),
\end{equation}
where $k_m$ is the time step on subinterval $I_m$ and $A_{k,m}$ is the average of $A(t)$ on $I_m$ (see section 2 for a detailed description).
In contrast to the continuous estimate \eqref{eq: continuous maximal parabolic}, the above estimates include  the limiting cases of $p=\infty$ and $p=1$, which explains the logarithmic factor in~\eqref{eq: maximal parabolic dgr intro combined} and \eqref{eq: maximal parabolic dgr L1}.

The corresponding results also hold for the fully discrete approximation $u_{kh}$. Thus in particular for  $1\le p\le \infty$ and $u_0=0$, we establish
\begin{equation}\label{eq: maximal parabolic dgr intro combined Lp fully}
\left[\sum_{m}\left(\|A_{kh,m} u_{kh,m}\|^p_{L^p(I_m;L^2(\Om))}+k_m\left\|\frac{[u_{kh}]_{m-1}}{k_m} \right\|^p_{L^2(\Om)}\right)\right]^{\frac{1}{p}}\le C\lk\|f\|_{L^p(I;L^2(\Om))}, 
\end{equation}
with corresponding changes for $p=\infty$.  We would like to point out that the above fully discrete result is valid on rather general meshes and does not require the mesh to be quasi-uniform or even shape regular, only admissible (no hanging nodes). 

As an application of the discrete maximal parabolic regularity we show that if the coefficients $a_{ij}(t,x)$ are sufficiently regular (see Assumption \ref{assumption})  and $\Om$ convex  we  obtain symmetric error estimate 
$$
\|u-u_k\|_{L^p(I;L^2(\Om))}\le C\lk \|u-\pi_k u\|_{L^p(I;L^2(\Om))}, \quad 1\le p<\infty,
$$
where $\pi_k$ is an interpolation into the space of piecewise constant in time functions defined in \eqref{eq: projection pi_k}.
For $p=\infty$ we can establish even the best approximation type result 
$$
\|u-u_k\|_{L^\infty(I;L^2(\Om))}\le C\lk \|u-\chi\|_{L^\infty(I;L^2(\Om))},
$$
for any $\chi$ in the subspace of piecewise constant in time functions, see Theorem \ref{thm: best approxim semidiscrete infty}. The corresponding fully discrete versions are
$$
\|u-u_{kh}\|_{L^p(I;L^2(\Om))}\le C\lk \left(\|u-\pi_k u\|_{L^p(I;L^2(\Om))}+\|u-R_{h}u\|_{L^p(I;L^2(\Om))}\right)
$$
and
$$
\|u-u_{kh}\|_{L^\infty(I;L^2(\Om))}\le C\lk \left(\|u- \chi\|_{L^\infty(I;L^2(\Om))}+\|u-R_{h}u\|_{L^\infty(I;L^2(\Om))}\right),
$$
for any $\chi$ in the subspace and $R_{h}(t)$ being the Ritz projection corresponding to $A(t)$. The rate of convergence  depends
of course on the regularity of $u$. 

The rest of the paper is organized as follows. In section \ref{sec: preliminaries} we introduce the discontinuous Galerkin method and some notation. Section \ref{sec: time discretization}, which is the central piece of the paper, consists of several parts. First we write the dG($0$)  approximate solution $u_k$ in a convenient form. Then we introduce a transform function $w_k$ that satisfies a similar equation, but with transform operators. Then in a  series of lemmas we show that the resulting  operators are bounded in certain norms.  
Finally in sections \ref{subsec: discrete maximum time} and \ref{sec: fully discrete} we establish semidiscrete and fully discrete maximal parabolic regularity in $L^p(I; L^2(\Om))$ norms, respectively. We conclude our paper with section \ref{sec: applications}, where we show how the above 
 discrete maximal parabolic regularity results can be used to establish symmetric and best approximation type error estimates for the problems on convex domains with coefficients  satisfying some additional assumptions. 


\section{Preliminaries} \label{sec: preliminaries}
First, we introduce the bilinear form $a \colon \R \times H^1_0(\Omega) \times H^1_0(\Omega)\to \R$ defined by
\begin{equation}\label{eq: bilinear form}
    a(t;u,v)=\int_\Om \sum_{i,j=1}^da_{ij}(t,x)\partial_i u(x)\partial_j v(x)dx.
\end{equation}
From $\{a_{ij}(t,x)\}_{i,j=1}^d\subset L^\infty(I\times \Om)$ one can see that for each $t\in I$ the bilinear form $a(t;\cdot,\cdot)$ is bounded
\begin{equation}\label{eq: bounded form}
    a(t;u,v)\le C\|u\|_{H^1(\Om)}\|v\|_{H^1(\Om)},
\end{equation}
from the uniform ellipticity assumption \eqref{eq: uniform ellipticity}, it is coercive
\begin{equation}\label{eq: coercivity form}
    a(t;u,u)\geq \Cell\|u\|^2_{H^1(\Om)},
\end{equation}
and from \eqref{eq: condition on time dependent coefficients} it follows that
\begin{equation}\label{eq: bounded form with omega}
    |a(t_1;u,v)-a(t_2;u,v)|\le C\omega(|t_1-t_2|)\|u\|_{H^1(\Om)}\|v\|_{H^1(\Om)}.
\end{equation}
In view of the homogeneous Dirichlet boundary conditions the $H^1$ norm is equivalent to the $H^1$ seminorm. For each $t \in \bar{I}$ this bilinear form defines an operator $A(t) \colon H^1_0(\Omega)\to H^{-1}(\Omega)$ by
\[
\langle A(t) u, v \rangle = a(t,u,v) \quad \text{for all}\quad u,v\in H^1_0(\Omega),
\]
where $\langle\cdot, \cdot \rangle$ is the duality pairing between $H^1_0(\Om)$ and $H^{-1}(\Om)$ spaces. 

To introduce the time discontinuous Galerkin discretization for the problem,
 we partition  $I =(0,T]$ into subintervals $I_m = (t_{m-1}, t_m]$ of length $k_m = t_m-t_{m-1}$, where $0 = t_0 < t_1 <\cdots < t_{M-1} < t_M =T$. The maximal and minimal time steps are denoted by $k =\max_{m} k_m$ and $k_{\min}=\min_{m} k_m$, respectively.
We impose the following conditions on the time mesh.
\begin{enumerate}[(i)]
  \item There are constants $c,\beta>0$ independent on $k$ such that
    \[
      k_{\min}\ge ck^\beta.
    \]
  \item There is a constant $\kappa>0$ independent on $k$ such that for all $m=1,2,\dots,M-1$
    \[
    \kappa^{-1}\le\frac{k_m}{k_{m+1}}\le \kappa.
    \]
  \item It holds $k\le\frac{1}{4}T$. 
  \end{enumerate}
  Similar assumptions are made, e.g., in~\cite{MeidnerD_RannacherR_VexlerB_2011}.
The semidiscrete space $X_k^0$ of piecewise constant functions in time is defined by
\[
X_k^0=\{u_{k}\in L^2(I;H^1_0(\Om)) :\ u_{k}|_{I_m}\in \Ppol{0}(I_m;H^1_0(\Om)), \ m=1,2,\dots,M \},
\]
where $\Ppol{0}(V)$ is the space of constant functions in time with values in a Banach space $V$.
We will employ the notation
\[
v^+_m:=\lim_{t\to 0^+} v(t_m+t), \quad v^+_-:=\lim_{t\to 0^+} v(t_m-t), \quad \text{and}\quad [v]_m = v^+_m - v^-_m,
\]
if these limits exist. For a function $v_k$ from $X_k^0$ we denote $v_{k,m}:=v_k|_{I_m}$ resulting in
\[
v_{k,m}^+=v_{k,m+1},\quad v_{k,m}^-=v_{k,m}, \quad\text{and}\quad [v_k]_m=v_{k,m+1}-v_{k,m},
\]
for $m=1,2,\dots,M-1$.

Next we define the following bilinear form
\begin{equation}\label{eq: bilinear form B}
 B(u,\varphi)=\sum_{m=1}^M \langle \partial_t u,\varphi \rangle_{I_m \times \Omega} + \sum_{i,j=1}^d(a_{ij}\partial_i u,\partial_j \varphi)_{\IOm}+\sum_{m=2}^M([u]_{m-1},\varphi_{m-1}^+)_\Om+(u_{0}^+,\varphi_{0}^+)_\Om,
\end{equation}
where $(\cdot,\cdot)_{\Omega}$ and $(\cdot,\cdot)_{I_m \times \Omega}$ are the usual $L^2$ space and space-time inner-products,
$\langle \cdot,\cdot \rangle_{I_m \times \Omega}$ is the duality product between $ L^2(I_m;H^{-1}(\Omega))$ and $ L^2(I_m;H^{1}_0(\Omega))$. 
Rearranging the terms in \eqref{eq: bilinear form B}, we obtain an equivalent (dual) expression for $B$,
\begin{equation}\label{eq:B_dual}
 B(u,\varphi)= - \sum_{m=1}^M \langle u,\partial_t \varphi \rangle_{I_m \times \Omega} +  \sum_{i,j=1}^d(a_{ij}\partial_i u,\partial_j \varphi)_{\IOm}-\sum_{m=1}^{M-1} (u_m^-,[\varphi]_m)_\Om + (u_M^-,\varphi_M^-)_\Om.
\end{equation}
We note, that for $u_k,\varphi_k \in \Xk$ the bilinear form~\eqref{eq: bilinear form B} simplifies to
\[
B(u_k,\varphi_k)=\sum_{i,j=1}^d(a_{ij}\partial_i u_k,\partial_j \varphi_k)_{\IOm}+\sum_{m=2}^M([u_k]_{m-1},\varphi_{k,m})_\Om+(u_{k,1},\varphi_{k,1})_\Om
\] 
and
$$
 B(u_k,\varphi_k)=   \sum_{i,j=1}^d(a_{ij}\partial_i u_k,\partial_j \varphi_k)_{\IOm}-\sum_{m=1}^{M-1} (u_{k,m}^-,[\varphi_k]_m)_\Om + (u_{k,M}^-,\varphi_{k,M}^-)_\Om.
$$
The dG($0$) semidiscrete (in time) approximation $u_k\in \Xk$ of \eqref{eq: heat equation} is defined as
\begin{equation}\label{eq: semidiscrete heat with RHS}
B(u_k,\varphi_k)=(f,\varphi_k)_{\IOm}+(u_0,\varphi_{k,1})_\Om \quad \text{for all }\; \varphi_k\in \Xk,
\end{equation}
and by the construction we have the following Galerkin orthogonality
\begin{equation}\label{eq: semidiscrete Galerkin orthgonality}
B(u-u_k,\varphi_k)=0 \quad \text{for all }\; \varphi_k\in \Xk.
\end{equation}
To rewrite the dG($0$) method as a time-stepping scheme we introduce the following notation. We define
$f_k\in \Xk$ by
\begin{equation}\label{eq: fm}
f_{k,m}=\frac{1}{k_m}\int_{I_m}f(t)dt,\quad m=1,2,\dots,M
\end{equation}
and $A_{k,m} \colon H^1_0(\Om)\to H^{-1}(\Om)$ by
\begin{equation}\label{eq: Am}
A_{k,m}=\frac{1}{k_m}\int_{I_m}A(t)dt,\quad  m=1,2,\dots,M.
\end{equation}
Thus, the dG(0) solution $u_k$ satisfies
\begin{equation}\label{eq: one step dG0 inhomogeneous}
\begin{aligned}
u_{k,1}+k_1A_{k,1}u_{k,1}&=u_0+k_1f_{k,1},\\
u_{k,m}+k_mA_{k,m} u_{k,m} &= u_{k,m-1}+k_mf_{k,m},\quad m=2,3,\dots,M.
\end{aligned}
\end{equation}
To use results known for the autonomous problems we rewrite this formula for some fixed $2 \le m \le M$ as
\[
\begin{aligned}
u_{k,1}+k_1A_{k,m}u_{k,1}&=u_0+k_1f_{k,1} + k_1(A_{k,m}-A_{k_1})u_{k,1} ,\\
u_{k,l}+k_lA_{k,m} u_{k,l} &= u_{k,l-1}+k_lf_{k,l}+ k_l(A_{k,m}-A_{k_l})u_{k,l} ,\quad l=2,3,\dots,m.
\end{aligned}
\]
Then using the representation formula for the dG(0) solution in the autonomous case, (cf. the proof of Theorem 2.1 in~\cite{LeykekhmanD_VexlerB_2016b}), we obtain the following representation
\begin{equation}\label{eq: expresion for uk_m}
u_{k,m}=\sum_{l=1}^{m-1}k_lR_{m,l}(A_{k,m}-A_{k,l})u_{k,l}+\sum_{l=1}^mk_lR_{m,l}f_{k,l}+R_{m,1}u_0,\quad m=1,2,\dots,M,
\end{equation}
where
\begin{equation}\label{eq: expresion for R_km}
R_{m,l}=\prod_{j=1}^{m-l+1}r(k_{m+1-j}A_{k,m})\quad \text{and}\quad r(z)=\frac{1}{1+z}.
\end{equation}
Throughout the paper we use a convention $\sum_{l=1}^0=0$.
Next we define three operators $Q\colon X_k^0\to X_k^0$, $L\colon X_k^0\to X_k^0$, and $D\colon L^2(\Om)\to \Xk$ by
\begin{equation}\label{eq: defintion of Q}
(Qg_k)_m=\sum_{l=1}^{m-1}k_lA_{k,m}R_{m,l}(A_{k,m}-A_{k,l})A^{-1}_{k,l}g_{k,l},\quad \text{for}\ g_k\in X_k^0,
\end{equation}
\begin{equation}\label{eq: defintion of L}
(Lf_k)_m=\sum_{l=1}^mk_lA_{k,m}R_{m,l}f_{k,l},
\end{equation}
and
\begin{equation}\label{eq: expresion for D_m}
(Du_0)_m=A_{k,m}R_{m,1}u_0.
\end{equation}
Thus, for $v_k\in \Xk$ defined by 
$$
v_{k,l}=A_{k,l}u_{k,l}\quad\text{for}\quad  l=1,2,\dots,M,
$$ 
we have
\begin{equation}\label{eq: v=Qv+Lf}
v_k=Qv_k+Lf_k+Du_0.
\end{equation}


\section{Maximal parabolic regularity for time discretization}\label{sec: time discretization}

\subsection{Estimate for the transformed operator}
Let $\mu>0$ be a sufficiently large number to be chosen later. Define $w_{k,m}$ by
$$
w_{k,m}=\prod_{l=1}^m(1+\mu k_l)^{-1}u_{k,m}\quad m=1,2,\dots,M.
$$
Thus using \eqref{eq: one step dG0 inhomogeneous} we obtain
$$
\begin{aligned}
(1+\mu k_1)w_{k,1}+ k_1 (1+\mu k_1)A_{k,1}w_{k,1} &= u_{0}+k_1f_{k,1},\\
\prod_{l=1}^m(1+\mu k_l)w_{k,m}+ k_m\prod_{l=1}^m(1+\mu k_l)A_{k,m} w_{k,m} &= \prod_{l=1}^{m-1}(1+\mu k_l)w_{k,m-1}+k_mf_{k,m},
\end{aligned}
$$
for $m=2,\dots,M$.
Dividing both sides of the last equation by $\prod_{l=1}^{m-1}(1+\mu k_l)$, we obtain
$$
(1+k_m\mu)w_{k,m}+k_m(1+k_m\mu)A_{k,m} w_{k,m} = w_{k,m-1}+\prod_{l=1}^{m-1}(1+\mu k_l)^{-1}k_mf_{k,m}.
$$
Hence, we can rewrite \eqref{eq: one step dG0 inhomogeneous} as
$$
\begin{aligned}
w_{k,1}+ k_1 \widetilde{A}_{k,1}w_{k,1} &= u_{0}+k_1\widetilde{f}_{k,1},\\
w_{k,m}+k_m\widetilde{A}_{k,m} w_{k,m} &= w_{k,m-1}+k_m\widetilde{f}_{k,m},\quad m=2,\dots,M,
\end{aligned}
$$
where 
\begin{equation}\label{eq: expression for tilde A and tilde f}
    \widetilde{A}_{k,m}=(1+k_m\mu)A_{k,m}+\mu\operatorname{Id},\quad \widetilde{f}_{k,m}=\prod_{l=1}^{m-1}(1+\mu k_l)^{-1}{f}_{k,m},\quad  m=1,2,\dots,M.
\end{equation}
Here we use a convention $\prod_{l=1}^0=1$.
Similarly to \eqref{eq: v=Qv+Lf}, for $\tilde{v}_k\in \Xk$ defined by 
$$
\tilde{v}_{k,l}=\widetilde{A}_{k,l}w_{k,l}\quad\text{for}\quad  l=1,\dots,M,
$$
 we have
\begin{equation}\label{eq: v=tilde Qw+tilde Lf}
    \tilde{v}_k=\widetilde{Q}\tilde{v}_k+\widetilde{L}\widetilde{f}_k+\widetilde{D}u_0,
\end{equation}
where similarly to the definitions of $Q$, $L$, and $D$ above,
    \begin{equation}\label{eq: expresion for tilde R_km}
\widetilde{R}_{m,l}=\prod_{j=1}^{m-l+1}r(k_{m+1-j}\widetilde{A}_{k,m}),
\end{equation}
\begin{equation}\label{eq: defintion of tilde Q}
(\widetilde{Q}g_k)_m=\sum_{l=1}^{m-1}k_l\widetilde{A}_{k,m}\widetilde{R}_{m,l}(\widetilde{A}_{k,m}-\widetilde{A}_{k,l})\widetilde{A}^{-1}_{k,l}g_{k,l}
\end{equation}
and
\begin{equation}\label{eq: defintion of tilde L}
(\widetilde{L}\widetilde{f}_k)_m=\sum_{l=1}^mk_l\widetilde{A}_{k,m}\widetilde{R}_{m,l}\widetilde{f}_{k,l},\quad (\widetilde{D}{u}_0)_m=\widetilde{A}_{k,m}\widetilde{R}_{m,1}u_0.
\end{equation}
Using the ellipticity of $A_{k,m}$ we obtain the following resolvent estimate for $\widetilde{A}_{k,m}$. For a given  a given $\gamma\in (0,\pi/2)$ we define 
\begin{equation}\label{eq: definition of sigma}
{\Sigma_{\gamma}}= \{z \in \mathbb{C} : \abs{\arg{(z)}} \le \gamma\}.
\end{equation}
Moreover we introduce the complex spaces $\mathbb{H}=L^2(\Om)+iL^2(\Om)$ and $\V=H^1_0(\Om)+iH^1_0(\Om)$.
\begin{lemma}\label{lemma: tilde resolvent}
For any $\gamma>0$, there exists a constant $C$ independent of $k$ and $\mu$ such that
$$
\|(z-\widetilde{A}_{k,m})^{-1}v\|_{L^2(\Om)}\le  \frac{C}{|z|+\mu}\|v\|_{L^2(\Om)},\quad \forall z\in  \C \setminus\Sigma_{\gamma},\quad \forall v\in\mathbb{H}
$$
and
$$
\|\widetilde{A}_{k,m} (z-\widetilde{A}_{k,m})^{-1}v\|_{L^2(\Om)}\le  C\|v\|_{L^2(\Om)},\quad \forall z\in  \C \setminus\Sigma_{\gamma},\quad \forall v\in\mathbb{H}.
$$
\end{lemma}
\begin{proof}
For an arbitrary $v \in \mathbb{H}$  we define
\[
g = -(z-\widetilde{A}_{k,m})^{-1}v\in \V,
\]
or equivalently
\begin{equation}\label{eq: weak form g}
-z(g,\varphi)+(1+k_m\mu)(A_{k,m} g,\varphi)+\mu(g,\varphi)=(v,\varphi), \quad \forall \varphi \in \mathbb{V},
\end{equation}
which existence and uniqueness follow from the Fredholm alternative.
In this proof $(\cdot,\cdot)$ denotes the Hermitian inner product, i.e.
$(v,\varphi)=\int_\Om v\bar{\varphi}\ dx.$ 

Taking $\varphi={g}$  we obtain
\begin{equation}\label{eq: norms for g}
-z\|g\|^2_{L^2(\Om)}+(1+k_m\mu)(A_{k,m} g,g)+\mu\|g\|^2_{L^2(\Om)}=(v,g).
\end{equation}
Since $\gamma\le \abs{\arg{z}}\le \pi$ and $ \Cell\|g\|^2_{H^1(\Om)}\le (A_{k,m} g,g)$ and is real, this equation is of the form
$$
e^{i\delta}a + b = c,\quad \text{with}\quad a,b > 0,\quad \gamma\le \abs{\delta}\le \pi,
$$
by multiplying it by $e^{-\frac{i\delta}{2}}$ and taking real parts, we have
$$
a + b \le \left(\cos\left(\frac{\delta}{2}\right)\right)^{-1}|(v,g)| \le \left(\sin\left(\frac{\gamma}{2}\right)\right)^{-1}|(v,g)| = C_\gamma|(v,g)|.
$$
From \eqref{eq: norms for g} we therefore conclude
$$
\left(\abs{z}+\mu\right)\|g\|^2_{L^2(\Om)}+\Cell(1+k_m\mu)\|g\|^2_{H^1(\Om)}\le C_\gamma\|g\|_{L^2(\Om)}\|v\|_{L^2(\Om)}, \quad\text{for}\ z\in \mathbb{C}\setminus \Sigma_{\gamma}.
$$
Thus, we have 
$$
\|g\|_{L^2(\Om)}\le \frac{C_\gamma}{\abs{z}+\mu}\|v\|_{L^2(\Om)},
$$
which establishes the first result. The second result follows from the identity 
$$
\widetilde{A}_{k,m} (z-\widetilde{A}_{k,m})^{-1}=-\operatorname{Id}+z(z-\widetilde{A}_{k,m})^{-1}.
$$
\end{proof}

\begin{lemma}\label{lemma: tilde A_m inv in L2 and H1}
There exists a constant $C$ independent of $k$ and $\mu$ such that
$$
\|(\widetilde{A}_{k,m})^{-1} v\|_{L^2(\Om)}\le  \frac{1}{\mu}\|v\|_{L^2(\Om)}
$$
and
$$
\|(\widetilde{A}_{k,m})^{-1} v\|_{H^1(\Om)}\le  \frac{C}{\sqrt{\mu}(1+k_m\mu)^{\frac{1}{2}}}\|v\|_{L^2(\Om)}.
$$
\end{lemma}
\begin{proof}
For an arbitrary $v \in L^2(\Om)$, we define
\[
g =(\widetilde{A}_{k,m})^{-1}v,
\]
or equivalently
\begin{equation}\label{eq: weak form g real}
(\widetilde{A}_{k,m}g,\varphi)=(1+k_m\mu)(A_{k,m} g,\varphi)+\mu(g,\varphi)=(v,\varphi), \quad \forall \varphi\in H^1_0(\Om).
\end{equation}
Taking $\varphi=g$ and using the coercivity  \eqref{eq: coercivity form}, we obtain
\begin{equation}\label{eq: norms for g real}
\Cell(1+k_m\mu)\|g\|^2_{H^1(\Om)}+\mu\|g\|^2_{L^2(\Om)}\le \|v\|_{L^2(\Om)}\|g\|_{L^2(\Om)}.
\end{equation}
From the estimate  above, we immediately conclude that 
\begin{equation}\label{eq: norms for g in L2}
\|g\|_{L^2(\Om)}\le \frac{1}{\mu}\|v\|_{L^2(\Om)}
\end{equation}
and using \eqref{eq: norms for g in L2}, we also obtain
$$
\Cell(1+k_m\mu)\|g\|^2_{H^1(\Om)}\le \|v\|_{L^2(\Om)}\|g\|_{L^2(\Om)}\le \frac{1}{\mu}\|v\|^2_{L^2(\Om)},
$$
from which the second estimate of the lemma follows. 
\end{proof}

We will also require the following result that estimates the difference $\widetilde{A}_{k,m}-\widetilde{A}_{k,l}$.
\begin{lemma}\label{lemma: tilde A_m - Al in H-1}
The exists a constant $C$ independent of $k$ and $\mu$  such that for $m\geq l$
$$
\|(\widetilde{A}_{k,m}-\widetilde{A}_{k,l}) v\|_{H^{-1}(\Om)}\le C\left((1+\mu\min\{k_l,k_m\})\omega(t_{m}-t_{l-1})+\mu|k_m-k_l|\right)\|v\|_{H^1(\Om)}.
$$
\end{lemma}
\begin{proof}
By duality we have
$$
\|(\widetilde{A}_{k,m}-\widetilde{A}_{k,l}) v\|_{H^{-1}(\Om)}=\sup_{w\in H^1_0(\Om), \atop \|w\|_{H^1(\Om)}\le 1} ((\widetilde{A}_{k,m}-\widetilde{A}_{k,l}) v,w)_\Om.
$$
For each such $w$, we have
$$
((\widetilde{A}_{k,m}-\widetilde{A}_{k,l}) v,w)_\Om =\mu\left( (k_m{A}_{k,m}- k_l{A}_{k,l}) v,w\right)_\Om+(({A}_{k,m}-{A}_{k,l})v,w)_\Om=J_1+J_2.
$$
Using the definitions of $A_{k,m}$ and $A_{k,l}$, changing variables, and using \eqref{eq: condition on time dependent coefficients} and that $\omega$ is nondecreasing, we have  we have
\begin{equation}\label{eq: Akm-Akl}
\begin{aligned}
J_2 &=\left(\Bigl(\frac{1}{k_m}\int_{t_{m-1}}^{t_m}A(t)dt-\frac{1}{k_l}\int_{t_{l-1}}^{t_l}A(t)dt\Bigr) v,w\right)_\Omega\\
&=\int_0^1\Bigl((A(sk_m+t_{m-1})-A(sk_l+t_{l-1}))v,w\Bigr)_\Omega ds\\
&\le \int_0^1\omega(|sk_m+t_{m-1}-sk_l-t_{l-1}|)\|v\|_{H^1(\Om)}\|w\|_{H^1(\Om)} ds\\
&\le \omega(t_{m}-t_{l-1})\|v\|_{H^1(\Om)}\|w\|_{H^1(\Om)}.
\end{aligned}
\end{equation}
To estimate $J_1$ we use
$$
k_m{A}_{k,m}- k_l{A}_{k,l}= k_m({A}_{k,m}-{A}_{k,l})+(k_m-k_l)A_{k,l}
$$
or 
$$
k_m{A}_{k,m}- k_l{A}_{k,l}= k_l({A}_{k,m}-{A}_{k,l})+(k_m-k_l)A_{k,m}.
$$
Then using \eqref{eq: Akm-Akl}, we obtain
$$
J_1\le C\mu\left(\min\{k_l,k_m\}\omega(t_{m}-t_{l-1})\|v\|_{H^1(\Om)}\|w\|_{H^1(\Om)}+|k_m-k_l|\|v\|_{H^1(\Om)}\|w\|_{H^1(\Om)}\right).
$$
Combining the estimates for $J_1$ and $J_2$, we obtain the lemma.
\end{proof}

\begin{lemma}\label{lemma:Laplace2}
There exists a constant $C$ independent of $k$ and $\mu$ such that
\[
\|\widetilde{A}_{k,m}\widetilde{R}_{m,l} v\|_{L^2(\Om)}\le \frac{C}{t_m-t_{l-1}}\|v\|_{L^2(\Om)},\quad \forall v\in L^2(\Om).
\]
Moreover,  for $m-l\geq 1$ there holds
\[
\|\widetilde{A}_{k,m}^2\widetilde{R}_{m,l} v\|_{L^2(\Om)}\le \frac{C}{(t_m-t_{l-1})^2}\|v\|_{L^2(\Om)},\quad \forall v\in L^2(\Om).
\]
\end{lemma}
\begin{proof}
First we observe that each term $\widetilde{R}_{m,l}v$ can be thought of as $m-l+1$ time steps of dG(0) method of the homogeneous problem
$$
\pa_tu+\widetilde{A}_{k,m}u=0
$$
with the initial condition $u(t_{l-1})=v$. Then, using Lemma \ref{lemma: tilde resolvent} the first estimate follows from~\cite{ErikssonK_JohnsonC_LarssonS_1998}, (cf. also~\cite[Theorem~1]{LeykekhmanD_VexlerB_2016b}). To prove the second estimate we use a representation
\[
\widetilde{A}_{k,m}^2\widetilde{R}_{m,l} v = g(\widetilde{A}_{k,m}) v
\]
with the function
\[
g(\lambda) = \lambda^2 \prod_{j=l}^{m}r(k_{j}\lambda).
\]
Using the fact that the spectrum $\sigma(\widetilde{A}_{k,m})$ of $\widetilde{A}_{k,m}$ is real and positive we obtain by the Parseval's identity (cf. \cite[Chap. 7]{ThomeeV_2006})
\[
\|\widetilde{A}_{k,m}^2\widetilde{R}_{m,l} v\|_{L^2(\Om)} \le \sup_{\lambda\in\sigma(\widetilde{A}_{k,m})}\abs{g(\lambda)}\; \|v\|_{L^2(\Om)}.
\]
To estimate $\abs{g(\lambda)}$ we proceed similar to the proof of Theorem 5.1 in \cite{ErikssonK_JohnsonC_LarssonS_1998} and observe
$$
\prod_{j=l}^{m}(1+k_{j}\lambda)\geq 1+\lambda\sum_{j=l}^m k_j+\frac{\lambda^2}{2}\left(\sum_{i,j=l\atop {i\neq j}}^m k_ik_j\right)=1+\lambda(t_m-t_{l-1})+\frac{\lambda^2}{2}\left(\sum_{i,j=l\atop {i\neq j}}^m k_ik_j\right).
$$
Let $k_{\max} = \max_{l\le j\le m}$ $k_{j}$ and  first consider the case $k_{\max} < (t_{m}-t_{l-1})/2$. We have
$$
(t_m-t_{l-1})^2=\left(\sum_{j=l}^m k_j\right)^2=\sum_{j=l}^m k^2_j+\sum_{i,j=l\atop {i\neq j}}^m k_ik_j\le k_{\max}(t_m-t_{l-1})+\sum_{i,j=l\atop {i\neq j}}^m k_ik_j,
$$
and with the assumption $k_{\max} < (t_{m}-t_{l-1})/2$, 
$$
\sum_{i,j=l\atop {i\neq j}}^m k_ik_j\geq \frac{(t_m-t_{l-1})^2}{2}.
$$
This results in
$$
\prod_{j=l}^{m}(1+k_{j}\lambda)\geq 1+\lambda(t_m-t_{l-1})+\frac{\lambda^2}{4}(t_m-t_{l-1})^2
$$
and therefore we have
\[
\abs{g(\lambda)} \le \frac{\lambda^2}{1+\lambda(t_m-t_{l-1})+\frac{\lambda^2}{4}(t_m-t_{l-1})^2} \le \frac{4}{(t_m-t_{l-1})^2},
\]
which proves the assertion in this case. 
In the case $k_{\max} \geq (t_{m}-t_{l-1})/2$ let $l\le m_0\le m$ be such that $k_{m_0}=k_{\max}$. Due to $m-l\ge 1$ we can choose $m'_0$ as either $m_0-1$ or $m_0+1$ such that $l\le m'_0\le m$. Then we obtain
\[
\begin{aligned}
\sup_{\lambda\in\sigma(\widetilde{A}_{k,m})}|\lambda^2\prod_{j=l}^{m}r(k_{j}\lambda)|
&\le \sup_{\lambda\in\sigma(\widetilde{A}_{k,m})}\left|\frac{\lambda^2}{(1+k_{m_0}\lambda)(1+k_{m'_0}\lambda)}\right|\\
&\le \frac{1}{k_{m_0}\, k_{m'_0}}\le \frac{C}{k^2_{\max}}\le \frac{C}{(t_m-t_{l-1})^2},
\end{aligned}
\]
where we have used our assumption $(ii)$ on the time steps. This completes the proof for this case.
\end{proof}

\begin{lemma}\label{lemma: tilde Am Rm from L2toH1}
There exists a constant $C$ independent of $k$ and $\mu$ such that for $m-l\geq 1$,
$$
\|\widetilde{A}_{k,m}\widetilde{R}_{m,l} v\|_{H^1(\Om)}\le  \frac{C}{(t_m-t_{l-1})^{\frac{3}{2}}(1+\mu k_m)^{\frac{1}{2}}}\|v\|_{L^2(\Om)},\quad \forall v\in L^2(\Om).
$$
\end{lemma}
\begin{proof}
By the coercivity of the operator $A$  for any $w\in H^1_0(\Om)$, we have
$$
(\widetilde{A}_{k,m}w,w)\geq (1+\mu k_m)(A_{k,m}w,w)\geq (1+\mu k_m)\Cell\|w\|^2_{H^1(\Om)}.
$$ 
Thus, with $w=\widetilde{A}_{k,m}\widetilde{R}_{m,l}v$, we have by the previous lemma
\begin{align*}
(1+\mu k_m)\Cell\|\widetilde{A}_{k,m}\widetilde{R}_{m,l}v\|^2_{H^1(\Om)}&\le (\widetilde{A}_{k,m}\widetilde{A}_{k,m}\widetilde{R}_{m,l}v,\widetilde{A}_{k,m}\widetilde{R}_{m,l}v)\\ 
&\le \|\widetilde{A}_{k,m}^2\widetilde{R}_{m,l}v\|_{L^2(\Om)}\|\widetilde{A}_{k,m}\widetilde{R}_{m,l}v\|_{L^2(\Om)}\\
&\le \frac{C}{(t_m-t_{l-1})^3}\|v\|_{L^2(\Om)}.
\end{align*}
This completes the proof.
\end{proof}

\begin{lemma}\label{lemma: tilde Am Rm from H-1toL2}
There exists a constant $C$ independent of $k$ and $\mu$ such that for $m-l\geq 1$
$$
\|\widetilde{A}_{k,m}\widetilde{R}_{m,l} v\|_{L^2(\Om)}\le  \frac{C}{(t_m-t_{l-1})^{\frac{3}{2}}(1+\mu k_m)^{\frac{1}{2}}}\|v\|_{H^{-1}(\Om)},\quad\forall v\in H^{-1}(\Om).
$$
\end{lemma}
\begin{proof}
By duality
$$
\|\widetilde{A}_{k,m}\widetilde{R}_{m,l} v\|_{L^2(\Om)}=\sup_{w\in L^2(\Om), \atop \|w\|_{L^2(\Om)}\le 1} (\widetilde{A}_{k,m}\widetilde{R}_{m,l} v,w)_\Om.
$$
Since $A$ is a symmetric operator, we have $\widetilde{A}_{k,m}=\widetilde{A}_{k,m}^*$ and as a result $\widetilde{R}_{m,l}=\widetilde{R}^*_{m,l}$. Moreover $\widetilde{A}_{k,m}$ and $\widetilde{R}_{m,l}$ commute. Thus,
$$
(\widetilde{A}_{k,m}\widetilde{R}_{m,l} v,w)_\Om=( v, \widetilde{A}^*_{k,m}\widetilde{R}^*_{m,l}w)_\Om\le \|v\|_{H^{-1}(\Om)}\|\widetilde{A}^*_{k,m}\widetilde{R}^*_{m,l}w\|_{H^{1}(\Om)}.
$$
Since $\widetilde{A}^*_{k,m}=\widetilde{A}_{k,m}$, by Lemma \ref{lemma: tilde Am Rm from L2toH1}, we obtain
$$
\|\widetilde{A}^*_{k,m}\widetilde{R}^*_{m,l}w\|_{H^{1}(\Om)}\le \frac{C}{(t_m-t_{l-1})^\frac{3}{2}(1+\mu k_m)^{\frac{1}{2}}}\|w\|_{L^2(\Om)},
$$
which establishes the lemma.
\end{proof}

Combining the above lemmas we obtain the following result.
\begin{lemma}\label{lemma: Id - Q}
There exist constants $C_1$ and $C_2$ independent of $\mu$ and $k$ such that for any $g_k \in \Xk$ 
$$
\|(\widetilde{Q}g_k)_m\|_{L^2(\Om)}\le \max_{1\le j\le m}\|g_{k,j}\|_{L^2(\Om)}\left(\frac{C_1}{\sqrt{\mu}}+C_2\sqrt{\mu}\sum_{l=1}^{m-1}\frac{k_l|k_m-k_l|}{(t_m-t_{l-1})^{\frac{3}{2}}}\right),\quad m=1,\dots,M
$$
and
$$
\sum_{m=1}^M k_m \|(\widetilde{Q}g_k)_m\|_{L^2(\Om)}\le \left(\sum_{l=1}^M k_l\|g_{k,l}\|_{L^2(\Om)}\right) \left(\frac{C_1}{\sqrt{\mu}}+C_2\sqrt{\mu}\max_{1\le l \le M}\sum_{m=l+1}^{M}\frac{k_m|k_m-k_l|}{(t_m-t_{l-1})^{\frac{3}{2}}}\right),
$$
where $\widetilde{Q}$ is the operator defined in \eqref{eq: defintion of tilde Q}.
\end{lemma}

\begin{proof}
Using that 
$$
(\widetilde{Q}g_k)_m=\sum_{l=1}^{m-1}k_l\widetilde{A}_{k,m}\widetilde{R}_{m,l}(\widetilde{A}_{k,m}-\widetilde{A}_{k,l})\widetilde{A}^{-1}_{k,l}g_{k,l},
$$
we have
$$
\|(\widetilde{Q}g_k)_m\|_{L^2(\Om)}\le \sum_{l=1}^{m-1} k_l\|\widetilde{A}_{k,m}\widetilde{R}_{m,l}\|_{ H^{-1}\to L^2}\|\widetilde{A}_{k,m}-\widetilde{A}_{k,l}\|_{ H^1\to H^{-1}}\|(\widetilde{A}_{k,l})^{-1}\|_{ L^2\to H^{1}}\|g_{k,l}\|_{L^2(\Om)}.
$$
Combining estimates from Lemma \ref{lemma: tilde A_m inv in L2 and H1}, Lemma \ref{lemma: tilde A_m - Al in H-1} and Lemma \ref{lemma: tilde Am Rm from H-1toL2}, and using that
$$
\frac{(1+\mu\min\{ k_l,k_m\})}{(1+\mu k_l)^{\frac{1}{2}}(1+\mu k_m)^{\frac{1}{2}}}\le 1,
 $$
for any $m=1,2,\dots,M$ we have 
\begin{equation}\label{eq:main lemma intermediate}
\|(\widetilde{Q}g_k)_m\|_{L^2(\Om)}\le C\sum_{l=1}^{m-1} \bigg( \frac{k_l}{\sqrt{\mu}}\frac{\omega(t_m-t_{l-1})}{(t_m-t_{l-1})^{\frac{3}{2}}}+\sqrt{\mu}k_l\frac{|k_m-k_l|}{(t_m-t_{l-1})^{\frac{3}{2}}} \bigg)\|g_{k,l}\|_{L^2(\Om)}.
\end{equation}
From the condition  \eqref{eq: condition on omega} and properties of the Riemann sums, we obtain
$$
\sum_{l=1}^{m-1}k_l\frac{\omega(t_m-t_{l-1})}{(t_m-t_{l-1})^{\frac{3}{2}}}\le \int_0^{t_m}\frac{\omega(t_m-s)}{(t_m-s)^{\frac{3}{2}}}\ ds\le C,
$$
 and taking the maximum over $l$ of $\|g_{k,l}\|_{L^2(\Om)}$, we obtain the first estimate of the lemma.
 
 From \eqref{eq:main lemma intermediate} we also obtain
 \[
\sum_{m=1}^M k_m \|(\widetilde{Q}g_k)_m\|_{L^2(\Om)}\le C\sum_{m=1}^M k_m \sum_{l=1}^{m-1} \left(\frac{k_l}{\sqrt{\mu}}\frac{\omega(t_m-t_{l-1})}{(t_m-t_{l-1})^{\frac{3}{2}}}
+\frac{\sqrt{\mu} k_l |k_m-k_l|}{(t_m-t_{l-1})^{\frac{3}{2}}}\right) \ltwonorm{g_{k,l}}.
\]
Changing the order of summation we have
\[
\sum_{m=1}^M k_m \|(\widetilde{Q}g_k)_m\|_{L^2(\Om)}\le \sum_{l=1}^{M-1} k_l \ltwonorm{g_{k,l}} \sum_{m=l+1}^M \left(\frac{k_m}{\sqrt{\mu}}\frac{\omega(t_m-t_{l-1})}{(t_m-t_{l-1})^{\frac{3}{2}}}
+\frac{\sqrt{\mu} k_m |k_m-k_l|}{(t_m-t_{l-1})^{\frac{3}{2}}}\right).
\]
Again similar to the above, by~\eqref{eq: condition on omega} we have,
\[
\sum_{m=l+1}^M k_m\frac{\omega(t_m-t_{l-1})}{(t_m-t_{l-1})^{\frac{3}{2}}} \le \int_{t_{l-1}}^T \frac{w(s-t_{l-1})}{(s-t_{l-1})^{\frac{3}{2}}} \, ds \le C.
\]
Taking maximum over $l$ in the sum $\sum_{m=l+1}^M  \frac{k_m |k_m-k_l|}{(t_m-t_{l-1})^{\frac{3}{2}}}$ completes the proof. 
\end{proof}

\begin{proposition}\label{prop:Q<1}
There exists $\mu>0$ sufficiently large and $\delta_0>0$ such that for  $k - k_{\min} \le \delta_0$
the following estimates hold for all $g_k\in \Xk$,
\begin{equation}\label{eq:est:prop:Linfty}
\|(\widetilde{Q}g_k)_m\|_{L^2(\Om)}\le \frac{3}{4}\max_{1\le l \le m} \|g_{k,l}\|_{L^2(\Om)},\quad m=1,2,\dots,M,
\end{equation}
and
\begin{equation}\label{eq:est:prop:L1}
\sum_{m=1}^M k_m \|(\widetilde{Q}g_k)_m\|_{L^2(\Om)}\le \frac{3}{4} \sum_{l=1}^M k_l \|g_{k,l}\|_{L^2(\Om)}.
\end{equation}

\end{proposition}
\begin{proof}
Using the first estimate from Lemma \ref{lemma: Id - Q} and choosing $\mu=4C_1^2$ we obtain
\[
\|(\widetilde{Q}g_k)_m\|_{L^2(\Om)}\le \max_{1\le j\le m}\|g_{k,j}\|_{L^2(\Om)} \left(\frac{1}{2} + C_2 \sqrt{\mu}\sum_{l=1}^{m-1}\frac{k_l\abs{k_m-k_l}}{(t_m-t_{l-1})^{\frac{3}{2}}}\right).
\]
Using $\abs{k_m-k_l} \le t_m-t_{l-1}$ we get for some $0<\eps<1$
\[
\begin{aligned}
C_2 \sqrt{\mu}\sum_{l=1}^{m-1}\frac{k_l\abs{k_m-k_l}}{(t_m-t_{l-1})^{\frac{3}{2}}}&\le  2 C_1 C_2 \sum_{l=1}^{m-1}\frac{k_l\abs{k_m-k_l}^{\frac{1}{2}-\eps}}{(t_m-t_{l-1})^{1-\eps}}\\
& \le 2 C_1 C_2 (k - k_{\min})^{\frac{1}{2}-\eps} \sum_{l=1}^{m-1}\frac{k_l}{(t_m-t_{l-1})^{1-\eps}}.
\end{aligned}
\]
Using  the properties of the Riemann sums we can estimate
\[
\sum_{l=1}^{m-1}\frac{k_l}{(t_m-t_{l-1})^{1-\eps}} \le \int_0^{t_{m-1}} \frac{1}{t_m-s} \le C_\eps.
\]
Choosing for example $\eps=\frac{1}{4}$ we get with $C_3 =  2 C_1 C_2 C_\eps$
\[
\|(\widetilde{Q}g_k)_m\|_{L^2(\Om)}\le\max_{1\le j\le m}\|g_{k,j}\|_{L^2(\Om)} \left( \frac{1}{2} + C_3 (k - k_{\min})^{\frac{1}{4}}\right).
\]
The estimate~\eqref{eq:est:prop:Linfty} follows then with the choice $\delta_0 = \frac{1}{(4C_3)^4}$.
The estimate~\eqref{eq:est:prop:L1} follows from Lemma~\ref{lemma: Id - Q}  similarly.
\end{proof}
\begin{remark}
The condition $k - k_{\min} \le \delta_0$ trivially holds in the case of uniform time steps. For non-uniform time steps it is sufficient to assume $k \le \frac{1}{2}\delta_0$.
\end{remark}

The above proposition shows that under certain conditions, the operator $\operatorname{Id}-\widetilde{Q}$ is invertible with a bounded inverse with respect to both the $L^\infty(I;L^2(\Omega))$ and $L^1(I;L^2(\Omega))$ norms on $\Xk$. This is the central piece of our argument. In order to obtain a discrete maximal parabolic regularity, we will also require estimates for $\widetilde{L}$ and $\tilde D$, which we will show next.

\begin{lemma}\label{lemma: tilde L}
For the operator $\widetilde{L}$ defined in~\eqref{eq: defintion of tilde L} there exists a constant $C$ independent of $k$ such that for all $f_k\in X_k^0$ the following estimates hold:
\begin{equation}\label{eq:est_tildeL:Linfty}
\|(\widetilde{L}f_k)_m\|_{L^2(\Om)}\le C\lk\max_{1\le l\le m}\|f_{k,l}\|_{L^2(\Om)}.
\end{equation}
and
\begin{equation}\label{eq:est_tildeL:L1}
\sum_{m_1}^M k_m \|(\widetilde{L}f_k)_m\|_{L^2(\Om)}\le C\lk \sum_{l=1}^M k_l \|f_{k,l}\|_{L^2(\Om)}.
\end{equation}
\end{lemma}
\begin{proof}
From the definition of $\widetilde{L}$ and Lemma~\ref{lemma:Laplace2} we obtain
$$
\begin{aligned}
\|(\widetilde{L}f_k)_m\|_{L^2(\Om)}&\le \sum_{l=1}^m k_l\|\widetilde{A}_{k,m}\widetilde{R}_{m,l}\|_{L^2\to L^2}\|f_{k,l}\|_{L^2(\Om)}\\
&\le C\max_{1\le l\le m}\|f_{k,l}\|_{L^2(\Om)}\sum_{l=1}^m \frac{k_l}{t_m-t_{l-1}}\le C\lk \max_{1\le l\le m}\|f_{k,l}\|_{L^2(\Om)},
\end{aligned}
$$
where in the last step we used that 
$$
\sum_{l=1}^m \frac{k_l}{t_m-t_{l-1}}\le 1+\int_0^{t_{m-1}} \frac{dt}{t_m-t}=1+\ln\frac{t_m}{k_m}\le C\lk.
$$
This completes the proof of~\eqref{eq:est_tildeL:Linfty}. The estimate~\eqref{eq:est_tildeL:L1} can be shown similarly by changing the order of summation.
\end{proof}

\begin{lemma}\label{lemma:tildeD}
For the operator $\widetilde{D}$ defined in~\eqref{eq: defintion of tilde L} there exists a constant $C$ independent of $k$ such that for all $u_0 \in L^2(\Omega)$ the following estimate holds,
\begin{equation}\label{eq:est_tildeD:L1}
\sum_{m=1}^M k_m \|(\widetilde{D} u_0)_m\|_{L^2(\Om)}\le C\lk \ltwonorm{u_0}.
\end{equation}
If in addition $u_0 \in H^1_0(\Omega)$ with $A_{k,m} u_0 \in L^2(\Omega)$ for all $m=1,2,\dots,M$ then
\begin{equation}\label{eq:est_tildeD:Linfty}
\max_{1\le m\le M} \|(\widetilde{D} u_0)_m\|_{L^2(\Om)} \le C \mu \max_{1\le m\le M} \ltwonorm{A_{k,m} u_0}.
\end{equation}
\end{lemma}
\begin{proof}
There holds by Lemma~\ref{lemma:Laplace2}
\[
\|(\widetilde{D} u_0)_m\|_{L^2(\Om)} \le \frac{C}{t_m} \ltwonorm{u_0}.
\]
This results in
\[
\sum_{m=1}^M k_m \|(\widetilde{D} u_0)_m\|_{L^2(\Om)}\le C \ltwonorm{u_0} \sum_{m=1}^M \frac{k_m}{t_m} \le C\lk \ltwonorm{u_0}.
\]
To prove~\eqref{eq:est_tildeD:Linfty} we use the fact that
\[
\|(\widetilde{D} u_0)_m\|_{L^2(\Om)} \le C \ltwonorm{\widetilde{A}_{k,m} u_0} \le C \mu \ltwonorm{A_{k,m} u_0}.
\]
\end{proof}

\subsection{Semidiscrete in time maximal parabolic regularity}\label{subsec: discrete maximum time}
Combining the above  results we can finally establish the maximal parabolic regularity with respect to the $L^\infty(I;L^2(\Om))$ and the $L^1(I;L^2(\Om))$ norms in the following two theorems.
\begin{theorem}[Discrete maximal parabolic regularity for $p=\infty$]\label{thm: maximal parabolic time dG0}
Let $f \in L^\infty(I;L^2(\Omega))$, let $u_0 \in H^1_0(\Omega)$ with $A_{k,m} u_0 \in L^2(\Omega)$ for all $m=1,2,\dots,M$. Let moreover $u_k$ be the dG(0) semidiscrete solution to \eqref{eq: semidiscrete heat with RHS}. 
There exists $\mu>0$ sufficiently large and $\delta_0>0$ such that for  $k - k_{\min} \le \delta_0$
\begin{multline*}
\max_{1 \le m \le M}\left(\|A_{k,m} u_{k,m}\|_{L^2(\Om)}+\left\|\frac{[u_{k}]_{m-1}}{k_m} \right\|_{L^2(\Om)}\right)\\\le Ce^{\mu T}\left(\lk\|f\|_{L^\infty(I;L^2(\Om))}+\mu T \max_{1\le m\le M} \ltwonorm{A_{k,m} u_0}\right).
\end{multline*}

\end{theorem}
\begin{proof}
Recalling the definitions of $\tilde{v}_k$ and $w_{k}$,  namely
$$
\tilde{v}_{k,m}=\widetilde{A}_{k,m}w_{k,m},\quad\text{and}\quad u_{k,m}=\prod_{l=1}^m(1+\mu k_l)w_{k,m},
$$
 we have
\begin{equation}\label{eq:calling u_km}
 \widetilde{A}_{k,m} u_{k,m} =\prod_{l=1}^m(1+\mu k_l)\widetilde{A}_{k,m}w_{k,m} =\prod_{l=1}^m(1+\mu k_l) \tilde{v}_{k,m}.
\end{equation}
By \eqref{eq: v=tilde Qw+tilde Lf} we have
\[
\max_{1 \le m \le M} \ltwonorm{\tilde{v}_{k,m}} \le \max_{1 \le m \le M} \left(\ltwonorm{(\widetilde{Q}\tilde{v}_k)_m} + \ltwonorm{(\widetilde{L}\tilde{f}_k)_m} + \ltwonorm{(\widetilde{D}u_0)_m}\right).
\]
For the first term we use estimate~\eqref{eq:est:prop:Linfty} from Proposition \ref{prop:Q<1}, for the second one estimate~\eqref{eq:est_tildeL:Linfty} from Lemma~\ref{lemma: tilde L} and for the third one estimate~\eqref{eq:est_tildeD:Linfty} from Lemma~\ref{lemma:tildeD}. This results in
\begin{multline*}
\max_{1 \le m \le M} \ltwonorm{\tilde{v}_{k,m}} \le \frac{3}{4} \max_{1\le m \le M} \|\tilde{v}_{k,m}\|_{L^2(\Om)} + C\lk\max_{1\le m\le M}\|\tilde f_{k,m}\|_{L^2(\Om)}\\ + C \mu T\max_{1\le m\le M} \ltwonorm{A_{k,m} u_0}.
\end{multline*}
Absorbing the first term on the right-hand side we obtain
\[
\max_{1 \le m \le M} \|\tilde{v}_{k,m}\|_{L^2(\Om)}\le C\lk\|\widetilde{f}_k\|_{L^\infty(I;L^2(\Om))}+C \mu T \max_{1\le m\le M} \ltwonorm{A_{k,m} u_0}.
\]
Now using that 
$$
\|\widetilde{f}_k\|_{L^\infty(I;L^2(\Om))}\le C\|f\|_{L^\infty(I;L^2(\Om))},
$$
we obtain
$$
\max_{1 \le m \le M} \|\widetilde{A}_{k,m} u_{k,m}\|_{L^2(\Om)}\le Ce^{\mu T}\left(\lk \|f\|_{L^\infty(I;L^2(\Om))}+ \mu T \max_{1\le m\le M} \ltwonorm{A_{k,m} u_0}\right),
$$
where we used that $\prod_{l=1}^m(1+\mu k_l)\le e^{\mu t_m}$.
Since $\tilde{A}_{k,m}$ is invertible for each $m$, using Lemma \ref{lemma: tilde A_m inv in L2 and H1} we also have
$$
\|u_{k,m}\|_{L^2(\Om)}=\|\widetilde{A}^{-1}_{k,m}\widetilde{A}_{k,m} u_{k,m}\|_{L^2(\Om)}\le
\frac{1}{\mu}\|\widetilde{A}_{k,m} u_{k,m}\|_{L^2(\Om)} .
$$
Thus from \eqref{eq:calling u_km}, the definition of $\widetilde{A}_{k,m}$, namely 
$$
\widetilde{A}_{k,m}=(1+k_m\mu)A_{k,m}+\mu\operatorname{Id}
$$
and by the triangle inequality and the estimates above we obtain 
\[
\|A_{k,m}u_{k,m}\|_{L^2(\Om)}\le \mu \|u_{k,m}\|_{L^2(\Om)}+\|\widetilde{A}_{k,m}u_{k,m}\|_{L^2(\Om)}\le 2\|\widetilde{A}_{k,m}u_{k,m}\|_{L^2(\Om)} 
\]
and therefore
\[
\max_{1 \le m \le M} \|A_{k,m} u_{k,m}\|_{L^2(\Om)}\le Ce^{\mu T} \left(\lk \|f\|_{L^\infty(I;L^2(\Om))}+ \mu T\max_{1\le m\le M} \ltwonorm{A_{k,m} u_0}\right).
\]
 The estimate for the second term in the statement of the theorem follows from the observation that \eqref{eq: one step dG0 inhomogeneous} is just
$$
\frac{[u_{k}]_{m-1}}{k_m} =-A_{k,m} u_{k,m}+f_{k,m}.
$$
\end{proof}

Next we establish the maximal parabolic regularity in $L^1(I;L^2(\Om))$ norm. 
\begin{theorem}[Discrete maximal parabolic regularity for $p=1$]\label{thm: maximal parabolic time dG0 in L1}
Let $f \in L^1(I;L^2(\Omega))$ and $u_0 \in L^2(\Omega)$. Let moreover $u_k$ be the dG(0) semidiscrete solution to \eqref{eq: semidiscrete heat with RHS}. 
There exists $\mu>0$ sufficiently large and $\delta_0>0$ such that for $k - k_{\min} \le \delta_0$ there holds
$$
\sum_{m=1}^M \left(k_m\|A_{k,m} u_{k,m}\|_{L^2(\Om)}+\|{[u_{k}]_{m-1}} \|_{L^2(\Om)}\right)\le Ce^{\mu T}\lk\left(\|f\|_{L^1(I;L^2(\Om))}+\|u_0\|_{L^2(\Om)}\right).
$$
\end{theorem}
\begin{proof}
The proof is very similar to the proof of the previous theorem. By \eqref{eq: v=tilde Qw+tilde Lf} we have
\[
\sum_{m=1}^M k_m \ltwonorm{\tilde{v}_{k,m}} \le \sum_{m=1}^M k_m \left(\ltwonorm{(\widetilde{Q}\tilde{v}_k)_m} + \ltwonorm{(\widetilde{L}\tilde{f}_k)_m} + \ltwonorm{(\widetilde{D}u_0)_m}\right).
\]
For the first term we use estimate~\eqref{eq:est:prop:L1} from Proposition \ref{prop:Q<1}, for the second one estimate~\eqref{eq:est_tildeL:L1} from Lemma~\ref{lemma: tilde L} and for the third one estimate~\eqref{eq:est_tildeD:L1} from Lemma~\ref{lemma:tildeD}. This results in
\begin{multline*}
\sum_{m=1}^M k_m \ltwonorm{\tilde{v}_{k,m}} \le \frac{3}{4} \sum_{m=1}^M k_m \|\tilde{v}_{k,m}\|_{L^2(\Om)} + C\lk \sum_{m=1}^M k_m \|\tilde f_{k,m}\|_{L^2(\Om)}\\ + C \lk \ltwonorm{u_0}.
\end{multline*}
The rest of the proof goes along the lines of the proof of the previous theorem.
\end{proof}

\begin{corollary}\label{cor: discrete maximal}
For $u_0=0$ interpolating between the results of Theorem \ref{thm: maximal parabolic time dG0} and  Theorem \ref{thm: maximal parabolic time dG0 in L1} we obtain the discrete maximal parabolic regularity for $1\le p< \infty$, namely
\begin{equation*}
\left[\sum_{m=1}^M\left(\|A_{k,m} u_{k,m}\|^p_{L^p(I_m;L^2(\Om))}+k_m\left\|\frac{[u_{k}]_{m-1}}{k_m} \right\|^p_{L^2(\Om)}\right)\right]^{\frac{1}{p}}\le Ce^{\mu T}\lk\|f\|_{L^p(I;L^2(\Om))}.
\end{equation*}
\end{corollary}

\subsection{Fully discrete maximal parabolic regularity}\label{sec: fully discrete}
In this section, we consider the fully discrete approximation of the equation \eqref{eq: heat equation}. We will establish the corresponding results for fully discrete approximations. 

 Let $\Om$ be a polygonal/polyhedral domain and let $\mathcal{T}$  denote  an admissible triangulation of $\Om$, i.e., $\mathcal{T} = \{\tau\}$ is a conformal partition of $\Om$ into  simplices (line segments, triangles, tetrahedrons, and etc.) $\tau$ of diameter $h_\tau$. Let  $h=\max_{\tau} h_\tau$ and $V_h$ be the set of all functions in $H^1_0(\Om)$ that are polynomials of degree $r\geq 1$ on each $\tau$, i.e., $V_h$ is the usual space of continuous finite elements. We would like to point out that we do not make any assumptions on shape regularity or quasi-uniformity of the meshes.
To obtain the fully discrete approximation we consider the space-time finite element space
\begin{equation} \label{def: space_time}
\Xkh=\{v_{kh} :\ v_{kh}|_{I_m}\in \Ppol{0}(I_m;V_h), \ m=1,2,\dots,M\}.
\end{equation}
We define a fully discrete analog $u_{kh} \in \Xkh$ of $u_k$ introduced in  \eqref{eq: semidiscrete heat with RHS} by
\begin{equation}\label{eq: semi fully discrete heat with RHS}
B(u_{kh},\varphi_{kh})=(f,\varphi_{kh})_{\IOm}+(u_0,\varphi_{kh}^+)_\Om \quad \text{for all }\; \varphi_{kh}\in \Xkh.
\end{equation}
Moreover, we introduce  two operators $A_{h}(t) \colon V_h \to V_h$ defined by 
\begin{equation}\label{eq: definition of Ah}
(A_{h}(t)v_h,\chi)_\Om=\sum_{i,j=1}^d(a_{ij}(t,\cdot)\pa_i v_h,\pa_j \chi)_\Om,\quad \forall \chi\in V_h
\end{equation}
and the orthogonal $L^2$ projection $P_h\colon V_h \to V_h$ defined by
$$
(P_{h}v_h,\chi)_\Om=( v_h, \chi)_\Om,\quad \forall \chi\in V_h.
$$ 
Similarly to $A_{k,m}$ in \eqref{eq: Am} we also define $A_{kh,m}\colon\Xkh\to \Xkh$
\begin{equation}\label{eq: Akhm}
A_{kh,m}=\frac{1}{k_m}\int_{I_m}A_h(t)dt,\quad  m=1,2,\dots,M.
\end{equation}

With the help of the above operators, the fully discrete approximation $u_{kh}\in \Xkh$ defined in \eqref{eq: semi fully discrete heat with RHS} satisfies 
\begin{equation}\label{eq: one step dG0 inhomogeneous fully}
\begin{aligned}
u_{kh,1}+k_1A_{kh,1}u_{kh,1}&=P_hu_0+k_1P_hf_{k,1},\\
u_{kh,m}+k_mA_{kh,m} u_{kh,m} &= u_{kh,m-1}+k_mP_hf_{k,m},\quad m=2,3,\dots,M,
\end{aligned}
\end{equation}  
where $f_{k,m}$ is defined in \eqref{eq: fm}.
Hence the same formulas for $u_k$, namely \eqref{eq: expression for uk_m}-\eqref{eq: v=Qv+Lf} also  hold for $u_{kh}$  with the difference that $A_k$ is replaced by $A_{kh}$ and $f_k$ by $P_hf_k$. The analysis from section \ref{sec: time discretization} of the paper translates almost immediately to the fully discrete setting since all arguments are  energy based arguments and  $\|P_h\|_{L^2(\Om)\to L^2(\Om)}\le 1$ on any mesh.  
Thus we directly establish the following results. 
\begin{theorem}[Discrete maximal parabolic regularity for $p=\infty$]\label{thm: maximal parabolic fully discrete}
Let the conditions of Theorem~\ref{thm: maximal parabolic time dG0} be fulfilled and let $u_{kh}$ be fully discrete solution to \eqref{eq: heat equation} defined by \eqref{eq: semi fully discrete heat with RHS} on any conformal triangulation of $\Om$.
Then there exists a constant $C$ independent of $k$ and $h$ such that the following estimate holds:
\begin{multline*}
\max_{1 \le m \le M}\left(\|A_{kh,m} u_{kh,m}\|_{L^2(\Om)}+\left\|\frac{[u_{kh}]_{m-1}}{k_m} \right\|_{L^2(\Om)}\right)\\\le Ce^{\mu T}\left(\lk\|f\|_{L^\infty(I;L^2(\Om))}+\mu T \max_{1\le m\le M} \ltwonorm{A_{kh,m} P_h u_0}\right).
\end{multline*}
\end{theorem}

The corresponding result for the $L^1(I;L^2(\Omega))$ norm is formulated in the following theorem.
\begin{theorem}[Discrete maximal parabolic regularity for $p=1$]\label{thm: maximal parabolic fully in L1}
Under the conditions of Theorem \ref{thm: maximal parabolic fully discrete}
there exists a constant $C$ independent of $k$ and $h$ such that
$$
\sum_{m=1}^M \left(k_m\|A_{kh,m} u_{kh,m}\|_{L^2(\Om)}+\|{[u_{kh}]_{m-1}} \|_{L^2(\Om)}\right)\le Ce^{\mu T}\lk\left(\|f\|_{L^1(I;L^2(\Om))}+\|u_0\|_{L^2(\Om)}\right).
$$
\end{theorem}

\begin{corollary}\label{cor: discrete maximal fully}
For $u_0=0$ interpolating between Theorem \ref{thm: maximal parabolic fully discrete} and  Theorem \ref{thm: maximal parabolic fully in L1}  we  obtain discrete maximal parabolic regularity for $1\le p< \infty$
\begin{equation*}
\left[\sum_{m=1}^M\left(\|A_{kh,m} u_{kh,m}\|^p_{L^p(I_m;L^2(\Om))}+k_m\left\|\frac{[u_{kh}]_{m-1}}{k_m} \right\|^p_{L^2(\Om)}\right)\right]^{\frac{1}{p}}\le Ce^{\mu T}\lk\|f\|_{L^p(I;L^2(\Om))}.
\end{equation*}
\end{corollary}

\section{Applications to error estimates}\label{sec: applications}

In this section we illustrate how the discrete maximal parabolic results from the previous section can be applied to error estimates.  For the rest of the section we assume that $\Om$ is convex and in addition the following assumption holds.
\begin{assumption}\label{assumption}
$$
a_{ij}(t,\cdot)\in W^{1,\infty}(\Om)\quad \text{for all $t\in \bar{I}$}, 
$$
and
$$
L:=\max_{1\le i,j\le d}\sup_{t\in \bar{I}}\|a_{ij}(t)\|_{W^{1,\infty}}<\infty.
$$
\end{assumption}

\subsection{Time semidiscrete  error estimates}\label{subsec: time error}
Using the convexity of $\Om$ and Assumption \ref{assumption} we establish the following preliminary result.
\begin{lemma}\label{lem: intermediate semidiscrete}
There exists a constant $C$ independent of  $k$ such that
$$
\sup_{t\in I_m}\|A(t)A^{-1}_{k,m}\|_{L^2\to L^2}\le C, \quad m=1,2,\dots,M.
$$
\end{lemma}
\begin{proof}
Take an arbitrary $v\in L^2(\Om)$ and set $w=A^{-1}_{k,m}v$, where  $A_{k,m}$ is an elliptic operator with coefficients $a_{k,m}^{ij}$ defined by
\begin{equation}\label{eq: definition of aijkm}
a_{k,m}^{ij}(x) = \frac{1}{k_m}\int_{I_m}a_{ij}(t,x)dt.
\end{equation}
By Assumption \ref{assumption} we have $a_{k,m}^{ij}\in W^{1,\infty}(\Om) $ with $\|a_{k,m}^{ij}\|_{W^{1,\infty}}\le L$ for all $1\le i,j\le d$.  From   Theorems 2.2.2.3 and 3.2.1.2 in \cite{GrisvardP_2011}, we can conclude that $w\in H^2(\Om)$ and
\begin{equation}\label{eq: w in H2}
\|w\|_{H^2(\Om)}\le C\|v\|_{L^2(\Om)},
\end{equation}
where the constant $C$ depends on $\Om$ and $L$ only. Again by Assumption \ref{assumption}, $A(t)$ is a bounded operator from $H^2(\Om)\cap H^1_0(\Om)$ to $L^2(\Om)$ and as a result
$$
\|A(t)A^{-1}_{k,m}v\|_{L^2(\Om)}=\|A(t)w\|_{L^2(\Om)}\le C\|w\|_{H^2(\Om)}\le C\|v\|_{L^2(\Om)}.
$$
Taking supremum over $v$ concludes the proof.
\end{proof}

As a first application of the discrete maximum regularity we establish semidiscrete best approximation result in the case of  $p=\infty$.   
\begin{theorem}\label{thm: best approxim semidiscrete infty}
Let the coefficients $a_{ij}(t,x)$ satisfy the Assumption \ref{assumption} and let $u$ be the solution to \eqref{eq: heat equation} with $u \in C(\bar I; L^2(\Omega))$ and $u_k$ be the dG(0) semidiscrete solution to \eqref{eq: semidiscrete heat with RHS}. Then under the conditions of Theorem \ref{thm: maximal parabolic time dG0 in L1} there exists a constant $C$ independent of $k$ such that
$$
\|u-u_k\|_{L^{\infty}(I;L^2(\Om))}\le Ce^{\mu T}\lk\inf_{\chi\in \Xk} \|u-\chi\|_{L^{\infty}(I;L^2(\Om))}.
$$
\end{theorem}
\begin{proof}
Let $\tilde{t} \in(0,T]$  be an arbitrary but fixed point in time. Without loss of generality we assume $\tilde{t} \in I_M=(t_{M-1},T]$.
We consider the following dual problem
\begin{equation}\label{eq: Green time only}
\begin{aligned}
\pa_tg(t,x)-A(t,x) g(t,x) &= \tilde{\theta}(t)u_k(\tilde{t},x), & (t,x) &\in \IOm,\;  \\
    g(t,x) &= 0,    & (t,x) &\in I\times\pa\Omega, \\
   g(T,x) &= 0,    & x &\in \Omega,
\end{aligned}
\end{equation}
where $\tilde{\theta}\in C^\infty(0,T)$ is the regularized Delta function in time with the properties $\supp(\tilde{\theta})\subset I_M$, $\|\tilde{\theta}\|_{L^1(I_M)}\le C$ and
$$
(\tilde{\theta},\varphi_k)_{I_M}=\varphi_k(\tilde{t}),\quad \forall \varphi_k\in \Xk.
$$ 
Let ${g}_{k}\in \Xk$ be dG($0$) approximation of ${g}$, i.e., $B({g}-{g}_{k},\varphi_{k})=0$ for any $\varphi_k\in \Xk$. 
Then 
$$
\begin{aligned}
\|u_{k}(\tilde{t})\|^2_{L^2(\Om)}&=(u_k(\cdot,\cdot),\tilde{\theta}(\cdot)u_k(\tilde{t},\cdot))_{I\times\Om}\\
&=B(u_k,g)=B(u_k,g_k)=B(u,g_k)\\
&=\sum_{m=1}^M(u,A(t)g_{k,m})_{I_m\times \Om}-\sum_{m=1}^M(u(t_m),[g_k]_m)_{\Om}:=J_1+J_2.
\end{aligned}
$$
Using the H\"{o}lder inequality in time and the Cauchy-Schwarz inequality in space, Theorem \ref{thm: maximal parabolic time dG0 in L1} and Lemma \ref{lem: intermediate semidiscrete}, we have 
$$
\begin{aligned}
J_1&\le \|u\|_{L^\infty(I;L^2(\Om))} \sum_{m=1}^Mk_m\sup_{t\in I_m}\|A(t)A_{k,m}^{-1}A_{k,m}g_{k,m}\|_{L^2(\Om)}\\
&\le C\|u\|_{L^\infty(I;L^2(\Om))} \sum_{m=1}^Mk_m\sup_{t\in I_m}\|A(t)A_{k,m}^{-1}\|_{L^2\to L^2}\|A_{k,m}g_{k,m}\|_{L^2(\Om)}\\
&\le Ce^{\mu T}\lk\|u\|_{L^\infty(I;L^2(\Om))} \|\tilde{\theta}\|_{L^1(I)}\|u_k(\tilde{t})\|_{L^2(\Om)}\\
&\le  Ce^{\mu T}\lk\|u\|_{L^\infty(I;L^2(\Om))}\|u_k(\tilde{t})\|_{L^2(\Om)}.
\end{aligned}
$$ 
Similarly, we obtain 
$$
\begin{aligned}
J_2& \le \sum_{m=1}^M \ltwonorm{u(t_m)} \ltwonorm{[g_k]_m}\\
&\le \|u\|_{L^\infty(I;L^2(\Om))} \sum_{m=1}^M\|[g_k]_m\|_{L^2(\Om)}\\
&\le  Ce^{\mu T}\lk\|u\|_{L^\infty(I;L^2(\Om))}\|u_k(\tilde{t})\|_{L^2(\Om)}.
\end{aligned}
$$ 
Canceling by $\|u_k(\tilde{t})\|_{L^2(\Om)}$ and taking the supremum over $\tilde{t}$, we establish
\begin{equation}\label{eq: after canceling uk}
\|u_{k}\|_{L^\infty(I;L^2(\Om))}\le Ce^{\mu T}\lk\|u\|_{L^\infty(I;L^2(\Om))}.
\end{equation}
Using  that the dG($0$) method is invariant on $\Xk$, by replacing $u$ and $u_{k}$ with $u-\chi$ and $u_{k}-\chi$ for any $\chi\in \Xk$, and using the triangle inequality we obtain the theorem.
\end{proof}

For $1\le p<\infty$, we can obtain the following result which is  similar to the result was obtained for time independent coefficients in \cite{LeykekhmanD_VexlerB_2016b} for the $L^p(I;L^s(\Om))$ norm. 
To state the result we define a projection $\pi_k$ for $u \in C(I,L^2(\Omega))$ with $\pi_k u|_{I_m} \in P_0(I_m;L^2(\Omega))$ for $m=1,2,\dots,M$ on each subinterval $I_m$ by
\begin{equation}\label{eq: projection pi_k}
\pi_k u(t)=u(t_m), \quad t\in I_m.
\end{equation}
\begin{theorem}\label{thm: semidiscrete best}
Let the coefficients $a_{ij}(t,x)$ satisfy the Assumption \ref{assumption} and let $u$ be the solution to \eqref{eq: heat equation} with $u \in C(\bar I; L^2(\Omega))$ and $u_k$ be the dG(0) semidiscrete solution to \eqref{eq: semidiscrete heat with RHS}. Then under the conditions of Theorem \ref{thm: maximal parabolic time dG0 in L1} there exists a constant $C$ independent of $k$ such that
$$
\|u-u_k\|_{L^p(I;L^2(\Om))}\le Ce^{\mu T}\lk \|u-\pi_k u\|_{L^p(I;L^2(\Om))},\quad 1\le p<\infty,
$$
where the projection $\pi_k$ is defined above in \eqref{eq: projection pi_k}.
\end{theorem}
\begin{proof}
The proof uses the result of Corollary~\ref{cor: discrete maximal} and goes along the lines of the proof of Theorem 9 in~\cite{LeykekhmanD_VexlerB_2016b} and Theorem \ref{thm: best approxim semidiscrete infty} above.
\end{proof}

\subsection{Applications to fully discrete error estimates}

Similarly to the semidiscrete case, as an application of the fully  discrete  maximal parabolic regularity, we show optimal convergence rates for the dG($0$)cG($r$) solution. As in the semidiscrete case, first using the convexity of $\Om$ and Assumption \ref{assumption} and in addition that the triangulation $\mathcal{T}$  is quasi-uniform we establish the space discrete version of the Lemma \ref{lem: intermediate semidiscrete}.
Thus, for rest of the paper  we assume
\begin{assumption}\label{assume: quasi}
There exists a constant $C$ independent of $h$ such that 
$$ 
\operatorname{diam}(\tau)\le h \le C |\tau|^{\frac{1}{d}}, \quad \forall \tau\in \mathcal{T},
$$
\end{assumption}
where $d=2,3$ is the dimension on $\Om$.
For the results below  we will require one Ritz projection $R_{h}(t)\colon H^1_0(\Omega)\to V_h$,  which is for every $t\in \bar{I}$ defined by 
\begin{equation}\label{eq: defition Ritz}
\sum_{i,j=1}^d\left(a_{ij}(t)\pa_i( R_h(t) v),\pa_j \chi\right)_\Omega=\sum_{i,j=1}^d\left(a_{ij}(t)\pa_i v,\pa_j \chi\right)_\Omega,\quad \forall \chi\in V_h
\end{equation}
and another 
Ritz projection $R_{kh,m}\colon H^1_0(\Omega)\to V_h$,  which is for every $m=1,2,\cdots,M$ defined by 
\begin{equation}\label{eq: defition Ritz km}
\sum_{i,j=1}^d\left(a_{k,m}^{ij}\pa_i R_{kh,m} v,\pa_j \chi\right)_\Omega=\sum_{i,j=1}^d\left(a_{k,m}^{ij}\pa_i v,\pa_j \chi\right)_\Omega,\quad \forall \chi\in V_h,
\end{equation}
where $a_{k,m}^{ij}$ are defined in \eqref{eq: definition of aijkm}.

\begin{lemma}\label{lem: intermediate fully discrete}
There exists a constant $C$ independent of $k$ and $h$ such that
$$
\sup_{t\in I_m}\|A_h(t)A^{-1}_{kh,m}\|_{L^2\to L^2}\le C, \quad m=1,2,\dots,M.
$$
\end{lemma}
\begin{proof}
Take an arbitrary $v\in L^2(\Om)$ and define $w_h=A^{-1}_{kh,m}P_hv$. In addition we also define $w=A^{-1}_{k,m}v$. Notice that $R_{kh,m}w = w_h$.
 By the definition of $A_h$ in \eqref{eq: definition of Ah},
$$
\left(A_h(t)w_h,\varphi_h\right)_\Om=\sum_{i,j=1}^d\left(a_{ij}(t)\pa_i w_h,\pa_j\varphi_h\right)_\Omega, \quad \forall \varphi_h\in V_h.
$$
 Put $z_h(t)=A_h(t)w_h$. Then adding and subtracting $w$, we have
 \begin{align*}
\|z_h(t)\|^2_{L^2(\Om)}&=\|A_h(t)w_h\|^2_{L^2(\Om)}= \sum_{i,j=1}^d\left(a_{ij}(t)\pa_i w_h,\pa_j z_h(t)\right)_\Omega\\
&= \sum_{i,j=1}^d\left(a_{ij}(t)\pa_i w,\pa_j z_h(t)\right)_\Omega+ \sum_{i,j=1}^d\left(a_{ij}(t)\pa_i (w_h-w),\pa_j z_h(t)\right)_\Omega:=J_1+J_2.
 \end{align*}
To estimate $J_1$ we integrate by parts and use the Cauchy-Schwarz inequality
$$
J_1 = -\sum_{i,j=1}^d\left(\pa_j (a_{ij}(t)\pa_i w),z_h(t)\right)_\Omega\le L\|w\|_{H^2(\Om)}\|z_h(t)\|_{L^2(\Om)}.
$$
Using that $w_h=R_{kh,m}w$, the standard elliptic error estimates and the inverse inequality,  we obtain
$$
\begin{aligned}
J_2 \le \sup_{i,j}\|a_{ij}(t)\|_{L^\infty(\Om)}\|w-w_h\|_{H^1(\Om)}\|z_h(t)\|_{H^1(\Om)}
&\le Ch\|w\|_{H^2(\Om)}\|z_h(t)\|_{H^1(\Om)}\\
&\le C\|w\|_{H^2(\Om)}\|z_h(t)\|_{L^2(\Om)}.
\end{aligned}
$$
Combining the estimates for $J_1$ and $J_2$ we can conclude that 
\begin{equation}\label{eq: estimate for zh}
\|z_h(t)\|_{L^2(\Om)}\le C\|w\|_{H^2(\Om)}\le C\|v\|_{L^2(\Om)},
\end{equation}
where we used that the definition of $w$ is identical to the definition of $w$ in Lemma \ref{lem: intermediate semidiscrete} and from \eqref{eq: w in H2} we know that $\|w\|_{H^2(\Om)}\le C\|v\|_{L^2(\Om)}$.   Taking supremum over $v$ concludes the proof.
\end{proof}

 Similar to the semidiscrete case, we also establish a corresponding result for $p=\infty$ in the fully discrete case. 
\begin{theorem}\label{thm: best approxim fully infty}
Let the coefficients $a_{ij}(t,x)$ satisfy the Assumption \ref{assumption} and let $u$  be the solution to \eqref{eq: heat equation} with $u \in C(\bar I; L^2(\Omega))$ and $u_{kh}$ be the dG(0)cG(r) solution for  $r\geq 1$ on a quasi-uniform triangulation  $\mathcal{T}$ with the coefficients $a_{ij}(t,x)$ satisfying the Assumption \ref{assumption}. Then under the assumption of Theorem  \ref{thm: maximal parabolic fully discrete} there exists a constant $C$ independent of $k$ and $h$ such that for $1\le p<\infty$,
$$
\|u-u_{kh}\|_{L^\infty(I;L^2(\Om))}\le Ce^{\mu T}\lk \left(\min_{\chi\in \Xkh}\|u- \chi\|_{L^\infty(I;L^2(\Om))}+\|u-R_{h}u\|_{L^\infty(I;L^2(\Om))}\right).
$$
\end{theorem}
\begin{proof}
As in the proof of Theorem \ref{thm: best approxim semidiscrete infty}, let $\tilde{t} \in(0,T]$  be an arbitrary but fixed point in time. Without loss of generality we assume $\tilde{t} \in I_M=(t_{M-1},T]$. Consider the following dual problem
\begin{equation}\label{eq: Green time fully}
\begin{aligned}
\pa_tg(t,x)-A(t,x) g(t,x) &= \tilde{\theta}(t)u_{kh}(\tilde{t},x), & (t,x) &\in \IOm,\;  \\
    g(t,x) &= 0,    & (t,x) &\in I\times\pa\Omega, \\
   g(T,x) &= 0,    & x &\in \Omega,
\end{aligned}
\end{equation}
where $\tilde{\theta}\in C^\infty(0,T)$ is the regularized Delta function in time with properties $\supp(\tilde{\theta})\subset I_M$, $\|\tilde{\theta}\|_{L^1(I_M)}\le C$ and
$$
(\tilde{\theta},\varphi_{kh})_{I_M}=\varphi_{kh}(\tilde{t}),\quad \forall \varphi_{kh}\in \Xkh.
$$ 
Let ${g}_{kh}$ be dG($0$)cG($r$) approximation of ${g}$, i.e., $B({g}-{g}_{kh},\varphi_{kh})=0$ for any $\varphi_{kh}\in \Xkh$. 
Then 
$$
\begin{aligned}
\|u_{kh}(\tilde{t})\|^2_{L^2(\Om)}&=(u_{kh},\tilde{\theta}u_{kh}(\tilde{t}))_{I\times\Om}\\
&=B(u_{kh},g)=B(u_{kh},g_{kh})=B(u,g_{kh})\\
&=\sum_{m=1}^M(R_h(t)u,A_h(t)g_{kh,m})_{I_m\times \Om}-\sum_{m=1}^M(u(t_m),[g_{kh}]_m)_{\Om}=J_1+J_2.
\end{aligned}
$$
Using the H\"{o}lder inequality in time and the Cauchy-Schwarz inequality in space, Lemma \ref{lem: intermediate fully discrete} and Theorem \ref{thm: maximal parabolic fully in L1}, we have 
$$
\begin{aligned}
J_1&\le \|R_{h}u\|_{L^\infty(I;L^2(\Om))} \sum_{m=1}^Mk_m\sup_{t\in I_m}\|A_h(t)A^{-1}_{kh,m}A_{kh,m}g_{k,m}\|_{L^2(\Om)}\\
&\le \|R_{h}u\|_{L^\infty(I;L^2(\Om))} \sum_{m=1}^Mk_m\sup_{t\in I_m}\|A_h(t)A^{-1}_{kh,m}\|_{L^2\to L^2}\|A_{kh,m}g_{k,m}\|_{L^2(\Om)}\\
&\le Ce^{\mu T}\lk\|R_{h}u\|_{L^\infty(I;L^2(\Om))} \|\tilde{\theta}\|_{L^1(I)}\|u_{kh}(\tilde{t},\cdot)\|_{L^2(\Om)}\\
&\le  Ce^{\mu T}\lk\|R_{h}u\|_{L^\infty(I;L^2(\Om))}\|u_{kh}(\tilde{t})\|_{L^2(\Om)}.
\end{aligned}
$$ 
Exactly as in the estimate of $J_2$ in Theorem \ref{eq: projection pi_k}, we obtain 
$$
J_2\le  Ce^{\mu T}\lk\|u\|_{L^\infty(I;L^2(\Om))}\|u_{kh}(\tilde{t})\|_{L^2(\Om)}.
$$ 
Thus canceling $\|u_{kh}(\tilde{t})\|_{L^2(\Om)}$ and taking supremum over $\tilde{t}$, we establish
\begin{equation}\label{eq: after canceling ukh}
\|u_{kh}\|_{L^\infty(I;L^2(\Om))}\le Ce^{\mu T}\lk\left(\|R_{h}u\|_{L^\infty(I;L^2(\Om))}+\|u\|_{L^\infty(I;L^2(\Om))} \right).
\end{equation}
Using  that  dG($0$)cG($r$) method is invariant on $\Xkh$, by replacing $u$ and $u_{kh}$ with $u-\chi$ and $u_{kh}-\chi$ for any $\chi\in \Xkh$, and using the triangle inequality we obtain the theorem.
\end{proof}

\begin{theorem}\label{thm: last}
Let the coefficients $a_{ij}(t,x)$ satisfy the Assumption \ref{assumption} and let $u$  be the solution to \eqref{eq: heat equation} with $u \in C(\bar I; L^2(\Omega))$ and $u_{kh}$ be the dG(0)cG(r) solution for  $r\geq 1$ on a quasi-uniform triangulation  $\mathcal{T}$ with the coefficients $a_{ij}(t,x)$ satisfying the Assumption \ref{assumption}. Then under the assumption of Theorem  \ref{thm: maximal parabolic fully discrete} there exists a constant $C$ independent of $k$ and $h$ such that for $1\le p<\infty$,
$$
\|u-u_{kh}\|_{L^p(I;L^2(\Om))}\le C\lk \left(\|u-\pi_k u\|_{L^p(I;L^2(\Om))}+\|u-R_{h}u\|_{L^p(I;L^2(\Om))}\right),
$$
where the projection $\pi_k$ is defined in \eqref{eq: projection pi_k} and $R_{h}$ in \eqref{eq: defition Ritz}.
\end{theorem}
\begin{proof}
The proof uses the result of Corollary~\eqref{cor: discrete maximal fully} and goes along the lines of the proof of Theorem 12 in~\cite{LeykekhmanD_VexlerB_2016b} and Theorem \ref{thm: best approxim fully infty} above.
\end{proof}

For sufficiently regular solutions, combining the above two theorems and using the approximation theory we immediately obtain an optimal order convergence result. 
\begin{corollary}
Under the assumptions of Theorem \ref{thm: last} and the regularity $u\in W^{1,p}(I;L^2(\Om)\cap L^p(I;H^{r+1}(\Om))$ for some $1\le p\le\infty$, there exists a constant $C$ independent of $k$ and $h$ such that
$$
\|u-u_{kh}\|_{L^p(I;L^2(\Om))}\le C\lk \left(k\|u\|_{W^{1,p}(I;L^2(\Om))}+h^{r+1}\|u\|_{L^{p}(I;H^{r+1}(\Om))}\right).
$$
\end{corollary}

\begin{remark}
The results of Theorems \ref{thm: best approxim semidiscrete infty}, \ref{thm: semidiscrete best}, \ref{thm: best approxim fully infty}, and \ref{thm: last} also hold for the elliptic operator of the form $A(t,x)=b(t)A(x)$, where $b(t)\in C^{\frac{1}{2}+\eps}(\bar{I})$ and $A(x)$ is the second order elliptic operator with bounded coefficients. In view of the uniform ellipticity condition \eqref{eq: uniform ellipticity}, we have $b(t)\geq b_0>0$ for some $b_0\in \R^+$, and as a result  Lemmas \ref{lem: intermediate semidiscrete} and \ref{lem: intermediate fully discrete} trivially hold without any additional assumptions, such as Assumptions \ref{assumption} and \ref{assume: quasi}. Thus, the results of the theorems hold for non-convex polygonal/polyhedral domains and on graded meshes. 
\end{remark}


\section*{Acknowledgments}
The authors would like to thank Dominik Meidner and Lucas Bonifacius for the careful reading of the manuscript and providing valuable suggestions that help to improve the presentation of the paper.


\bibliography{biblio}
\bibliographystyle{siam}

\end{document}